\documentclass[a4paper,10pt]{article}
\usepackage{amsfonts}
\usepackage{bbm}
\usepackage{mathrsfs}
\usepackage{amsmath,amsthm,amssymb,amscd}
\usepackage[center]{titlesec}
\newtheorem{theo}{Theorem}[section]
\newtheorem{prop}[theo]{Proposition}
\newtheorem{lem}[theo]{Lemma}
\newtheorem{exam}[theo]{Example}
\newtheorem{rem}[theo]{Remark}

\newtheorem{cla}[theo]{Claim}
\newtheorem{defi}[theo]{Definition}

\newcommand{\bp}{\begin{proof}}
\newcommand{\ep}{\end{proof}}

\begin{document}
 \setlength{\baselineskip}{13pt} \pagestyle{myheadings}

 \title{
 {A note on the deformations of almost complex structures on closed four-manifolds}
 \thanks{
 Supported by NSFC (China) Grants 11471145, 11401514 (Tan), 11371309 (Wang),
 11426195 (Zhou).}
 }
 \author{{\large Qiang Tan\thanks {E-mail:
 tanqiang1986@hotmail.com }, Hongyu Wang, Jiuru Zhou}\\
 }
 \date{}
 \maketitle
 \noindent {\bf Abstract:}
  In this paper, we calculate the dimension of the $J$-anti-invariant cohomology subgroup $H_J^-$ on $\mathbb{T}^4$.
  Inspired by the concrete example, $\mathbb{T}^4$, we get that:
  On a closed symplectic $4$-dimensional manifold $(M, \omega)$, $h^-_J=0$ for generic $\omega$-compatible almost complex structures.
 \\

 \noindent {{\bf AMS Classification (2000):} 53C55, 53D05.}\\


 \section{Notations and main result}

  Let $M$ be a closed oriented smooth $2n$-manifold.
  An almost complex structure on $M$ is a differentiable endomorphism on the tangent bundle
  $$ J: TM\rightarrow TM\,\,\,{\rm with}\,\,\,J^2=-id.$$
  Suppose $(M,J)$ is a closed almost complex manifold.
  One can construct a $J$-invariant Riemannian metric $g$ on $M$.
  Such a metric $g$ is called an almost Hermitian metric for $(M,J)$.
  We must point out that the $J$-invariant Riemannian metric always exists,
  for example, we can construct $g$ by
  $$g(\cdot,\cdot)=\frac{1}{2}(h(\cdot,\cdot)+h(J\cdot,J\cdot))$$
  for any Riemannian metric $h(\cdot,\cdot)$.
  This then in turn gives a $J$-compatible non-degenerate 2-form $F$ by $F(X,Y)=g(JX,Y)$,
  called the fundamental 2-form.
  Such a quadruple $(M,g,J,F)$ is called a closed almost Hermitian manifold.
  Thus an almost Hermitian structure on $M$ is a triple $(g,J,F)$.
  If $dF=0$, then $F$ will be written as $\omega$ and $(M,g,J,\omega)$ is called an almost K\"{a}hler manifold.

  \vskip 6pt

  Note that $J$ acts on the space $\Omega^2$ of 2-forms on $M$ as an involution by
  \begin{align}
  \alpha \longmapsto \alpha(J\cdot,J\cdot), \quad \alpha\in\Omega^2(M).
  \end{align}
  This gives the $J$-invariant, $J$-anti-invariant decomposition of 2-forms:
  \begin{align}
  \Omega^2 = \Omega^+_J \oplus \Omega^-_J, \quad \alpha = \alpha_J^+ + \alpha_J^- %
  \end{align}
  as well as the splitting of corresponding vector bundles
  \begin{align}\label{splitting of bundles}
  {\Lambda}^2={\Lambda}_J^+ \oplus {\Lambda}_J^-.
  \end{align}

  \begin{defi}
  Let $\mathcal Z^2$ denote the space of closed 2-forms on $M$ and define
  \begin{align*}
  \mathcal Z_J^+ \triangleq\mathcal Z^2 \cap \Omega_J^+, \quad
  \mathcal Z_J^- \triangleq \mathcal Z^2 \cap \Omega_J^-.
  \end{align*}
  \end{defi}

  It is well known that, when $J$ is integrable,
   $\beta\in \mathcal Z_J^-$ if and only if $J\beta\in \mathcal Z_J^-$.
  Conversely, if $(M,J)$ is a connected almost complex $4$-manifold and
  there exists nonzero $\beta\in \mathcal Z_J^-$ such that $J\beta\in \mathcal Z_J^-$,
  then $J$ is integrable (see \cite{Sa}).

  For an almost complex manifold $(M,J)$,
  T.-J. Li and W. Zhang \cite{LiZh} introduced subgroups, $H_{J}^{+}$ and $H_{J}^{-}$,
  of the real degree 2 de Rham cohomology group $H^2(M;\mathbb R)$, as the sets of cohomology classes
  which can be represented by $J$-invariant and $J$-anti-invariant real $2$-forms, respectively.

  \begin{defi}\label{definition of J anti}{\rm (cf. \cite{DLZ1,LiZh}) }
   Define the $J$-invariant and $J$-anti-invariant cohomology subgroups $H_J^{\pm}$ by
  \begin{align*}
  H^\pm_J=\{\mathfrak a \in H^2(M;\mathbb{R}) \mid %
           \mbox{there exists} \ \alpha \in \mathcal Z_J^{\pm} \ \mbox{such that} \ \mathfrak a = [\alpha]\}.
  \end{align*}
  We say $J$ is $C^\infty$  pure if $H_J^ + \cap H_J^- = \{0\}$, $C^\infty$  full
  if $H_J^+ + H_J^- = H^2(M;\mathbb{R})$, %
  and $J$ is $C^\infty$  pure and full if %
  \begin{align*}
  H^2(M;\mathbb{R}) = H_J^+\oplus H_J^-.
  \end{align*}
  Let us denote by $h_J^{+}$ and $h_J^{-}$ the dimensions of $H_{J}^{+}$ and $H_{J}^{-}$, respectively.
  \end{defi}

  It is interesting to consider whether or not the subgroups $H_{J}^{+}$ and $H_{J}^{-}$
  induce a direct sum decomposition of $H^2(M,\mathbb R)$.
  This is known to be true for integrable almost complex structures $J$
  which admit compatible K\"{a}hler metrics on compact manifolds of any dimension.
  In this case, the induced decomposition is nothing but the classical real
  Hodge-Dolbeault decomposition of $H^2(M,\mathbb R)$ (see \cite{Barth}).
  However, for non-integrable case, this is true only for dimension $4$.
  This is proved by T. Draghici, T.-J. Li and W. Zhang  in \cite{DLZ1}.

  \begin{prop} {\rm (cf. \cite[Theorem 2.3]{DLZ1})}\label{pure and full}
  For any closed almost complex 4-manifold $(M,J)$, $J$ is $C^\infty$ pure and full.
  \end{prop}

  Suppose $(M,g,J,F)$ is a closed almost Hermitian $4$-manifold, the Hodge star operator $*_g$ gives
  the well-known self-dual, anti-self-dual decomposition of 2-forms as well as the corresponding
  splitting of the bundle (see \cite{DK}):
  \begin{equation}\label{decomposition about g 1}
   \Omega^2 = \Omega_g^+ \oplus \Omega_g^-, \quad \alpha = \alpha_g^+ + \alpha_g^-;
  \end{equation}
  \begin{equation}\label{decomposition about g 2}
  {\Lambda}^2 = {\Lambda}_g^+ \oplus {\Lambda}_g^-.
  \end{equation}
  Since the Hodge-de Rham-Laplace operator commutes with $*_g$,
  the decomposition \eqref{decomposition about g 2} holds for the space $\mathcal {H}_g$ of harmonic 2-forms as well.
  By Hodge theory, this induces cohomology decomposition by the metric $g$:
  \begin{align}\label{harmonic decomposition}
  H^2(M;\mathbb{R})\cong\mathcal {H}_g=\mathcal {H}_g^+\oplus\mathcal
  {H}_g^-.
  \end{align}
  Similar to Definition \ref{definition of J anti}, one defines
  \begin{align}
  H_g^{\pm} = \{ \mathfrak a \in H^2(M;\mathbb{R}) \mid
             \mathfrak a = [\alpha] \,\; \mbox{for some} \,\;\alpha \in \mathcal Z_g^{\pm}:=
  \mathcal Z^2\cap\Omega_g^{\pm} \}.
  \end{align}
  It is easy to see that
  \begin{align*}
  H_g^\pm \cong \mathcal Z_g^{\pm} = \mathcal {H}_g^\pm
  \end{align*}
  and \eqref{harmonic decomposition} can be written as
  \begin{align}
  H^2(M;\mathbb{R}) = H_g^+ \oplus H_g^-.
  \end{align}

   On an almost Hermitian $4$-manifold,
   decompositions \eqref{splitting of bundles}
  and \eqref{decomposition about g 2} are related as follows:
  \begin{align}
  & {\Lambda}_J^+=\mathbb{R} F\oplus{\Lambda}_g^-, \label{205} \\
  & {\Lambda}_g^+=\mathbb{R} F\oplus{\Lambda}_J^-, \label{206} \\
  & {\Lambda}_J^+\cap{\Lambda}_g^+=\mathbb{R} F, \ \ \ \ {\Lambda}_J^-\cap{\Lambda}_g^-=\{0\}. \label{207}
  \end{align}
  It is easy to see that
  $\mathcal{Z}_J^-\subset \mathcal {H}_g^+$ and $\mathcal {H}_g^-\subset \mathcal{Z}_J^+$.
  Let $b_2$, $b^{+}$ and $b^{-}$ be the second, the self-dual and
  the anti-self-dual Betti number of $M$, respectively.
  Thus $b_2=b^{+}+b^{-}$. %
  Moreover, there hold (see \cite{DLZ1}):
  \begin{align}
  H_J^- \cong \mathcal Z_J^-, \quad
  h_J^+ + h_J^- = b_2, \quad
  h_J^+ \geq b^-, \quad h_J^- \leq b^+.
  \end{align}

  Lejmi recognizes $\mathcal Z_J^-$ as the kernel of an elliptic operator on $\Omega_J^-$.
  \begin{lem}\label{Lejmi lemma}{\rm (cf. \cite{Le1,Le2})}
  Let $(M,g,J,F)$ be a closed almost Hermitian $4$-manifold.
  Let operator $P: \Omega_J^-\rightarrow\Omega_J^-\nonumber$ be defined by
  \begin{align*}
   P(\psi) = P_J^-(d\delta_g\psi),
  \end{align*}
  where $P_J^-: \Omega^2 \rightarrow \Omega_J^-$ is the projection, $\delta_g$ is the
  codifferential operator with respect to metric $g$.
  Then $P$ is a self-adjoint strongly elliptic linear operator with kernel the $g$-harmonic $J$-anti-invariant 2-forms.
  \end{lem}

  Hence one has the decomposition of $\Omega_J^-$:
  \begin{align*}
  \Omega_J^- = {\rm ker}\, P \oplus P_J^-(d\Omega^1) = \mathcal{Z}_J^- \oplus P_J^-(d\Omega^1).
  \end{align*}

  \begin{rem}
  As a classical result of Kodaira and Morrow ({\rm \cite[Theorem 4.3]{KoMo}})
  showing the upper semi-continuity of the kernel of a family of elliptic differential operators,
  we know that $h^-_J$ is a upper semi-continuous function under the deformation of almost complex structures.
  \end{rem}

  Let $H_J^{-,\perp}$ denote the subgroup of $\mathcal {H}_g^+$ which
  is orthogonal to $H_J^-$ with respect to the cup product, that is,
  \begin{align}
  H_J^{-,\perp} := \{ \beta \in \mathcal Z_g^+ \mid
    \textstyle\int_M \beta\wedge\alpha=0 \ \ \forall \alpha \in \mathcal Z_J^- \}.
  \end{align}

  By Lemma \ref{Lejmi lemma} and the results in \cite[Lemmas 2.4 and 2.6]{DLZ1}, we have
  the following lemma.

  \begin{lem}\label{g lemma}
  ({\rm \cite{TWZZ}})
  Let $(M,g)$ be a closed Riemannian $4$-manifold.
  If $\alpha \in \Omega_g^+$ and
  $\alpha = \alpha_{\rm h} + d\theta + \delta_g\psi$ is its Hodge decomposition,
  then $P_g^+(d\theta) =  P_g^+(\delta_g\psi)$ and $P_g^-(d\theta) = -P_g^-(\delta_g\psi)$,
  where $P_g^{\pm}: \Omega^2 \rightarrow \Omega_g^{\pm}$ are the projections.
  Moreover, the $2$-form $\alpha-2P_g^+(d\theta)=\alpha_{\rm h}$ is harmonic and
  $\alpha + 2P_g^-(d\theta) = \alpha_{\rm h} + 2d\theta$.
  In particular, if $(M,g,J,F)$ is a closed almost Hermitian $4$-manifold and if $\alpha \in H_J^{-,\perp}$
  is a self-dual harmonic $2$-form, then $\alpha = fF + P_J^-(d\theta)$
  for some function $f \not\equiv 0$ and $\alpha-d\theta\in \mathcal Z_J^+$.
  \end{lem}

  \begin{rem}\label{g rem}
  As direct consequences of Lemmas \ref{Lejmi lemma} and \ref{g lemma}, we have decompositions
  as self-dual harmonic 2-forms and as cohomology classes:
  \begin{align*}
  \mathcal {H}_g^+ = \mathcal{Z}_J^- \oplus H_J^{-,\perp}, \quad\quad
          H_J^{+} \cong H_J^{-,\perp} \oplus \mathcal {H}_g^-.
  \end{align*}
  \end{rem}

  \vskip 6pt

   Almost complex structures $J_1$ and $J_2$ are said to be $g$-related if they are both compatible with $g$.
  In \cite{DLZ2}, T. Draghici, T.-J. Li and W. Zhang computed the subgroups $H_{J}^{+}$ and $H_{J}^{-}$
  and their dimensions $h_J^{+}$ and $h_J^{-}$ for almost complex structures metric related to an integrable one.
  Using Gauduchon metrics (\cite{Gauduchon}),
  they proved that the almost complex structures $\tilde{J}$ with $h_{\tilde{J}}^{-}=0$
  form an open dense set in the $C^{\infty}$-Fr\'{e}chet-topology in the space
   of almost complex structures metric related to an integrable one (\cite[Theorem 1.1]{DLZ2}).
  Based on this, they made a conjecture (Conjecture 2.4 in \cite{DLZ2}) about the dimension $h_J^-$ of
  $H_J^-$ on a compact $4$-manifold which asserts that $h_J^-$ vanishes for
  generic almost complex structures $J$.
  In particular, they have confirmed their conjecture for 4-manifolds with $b^+=1$
  (\cite[Theorem 3.1]{DLZ2}).
  Fortunately, in \cite{TWZZ},
  Qiang Tan, Hongyu Wang, Ying Zhang and Peng Zhu confirmed the conjecture completely.

  \begin{prop}\label{JGEA}
  ({\rm cf. \cite[Theorem 1.1]{TWZZ}})
  Let $M$ be a closed $4$-manifold admitting almost complex structures.
  Then the set of almost complex structures $J$ on $M$ with $h_J^-=0$
  is an open dense subset of $\mathcal{J}$ in the $C^\infty$-topology.
  \end{prop}

  A symplectic structure on a differentiable manifold is a nondegenerate closed $2$-form $\omega\in\Omega^2$.
  A differentiable manifold with some fixed symplectic structure  is called a symplectic manifold.
  Suppose $(M, \omega)$ is a closed symplectic manifold.
  Let $\mathcal{J}$ be the space of all almost complex structures on $M$
  and denote by $\mathcal{J}_{\omega}^{\rm c}$ and $\mathcal{J}_{\omega}^{\rm t}$ respectively
  the spaces of $\omega$-compatible and $\omega$-tamed almost complex structures on $M$.
  \begin{align*}
  & \mathcal{J}_{\omega}^{\rm t} = \{ J\in\mathcal{J} \mid \omega(X,JX)>0, \forall X \in TM, X \neq 0 \}, \\
  & \mathcal{J}_{\omega}^{\rm c} = \{ J\in\mathcal{J}_{\omega}^{\rm t} \mid \omega(JX,JY)=\omega(X,Y),
  \forall X,Y \in TM \}.
  \end{align*}
  It is well known that $\mathcal{J}_{\omega}^{\rm c}$ and $\mathcal{J}_{\omega}^{\rm t}$ are
  contractible $C^\infty$-Fr\'{e}chet spaces and $\mathcal{J}_{\omega}^{\rm t}$ is an
  open subset of $\mathcal{J}$ in the $C^\infty$-topology.
  See \cite{Au, DLZ1, DLZ2} for details.
  Then by Proposition \ref{JGEA}, we know that
   the set of $\omega$-tamed almost complex structures $J$ on $M$ with $h^-_J=0$ is an open
  dense subset of $\mathcal{J}_{\omega}^{\rm t}$ in the $C^\infty$-topology.
  In this paper we want to prove the compatible case:

  \vskip 6pt

  \noindent {\bf Main Theorem:}
 {\it Let $(M,\omega)$ be a closed symplectic $4$-manifold.
  Then the set of $\omega$-compatible almost complex structures $J$ on $M$ with $h^-_J=0$ is an open
  dense subset of $\mathcal{J}_{\omega}^{\rm c}$ in the $C^\infty$-topology.
  }

  \begin{rem}
  In general, on a closed almost K\"{a}hler $4$-manifold $(M,g,J,\omega)$, we have
   $0\leq h^-_J\leq b^+-1$.
   If $h^-_J= b^+-1$, one has a generalized $dd^c$-lemma (cf. {\rm \cite{Le3}}),
   where the twisted differential $d^c$ is defined by $d^c=(-1)^pJdJ$ acting on $p$-forms.
   Hence, we pose the following question:
   On a closed symplectic $4$-manifold $(M,\omega)$, is there a $\omega$-compatible almost complex structure $J$
   such that $h^-_J=b^+-1$ ?
  \end{rem}

  The rest of the paper is organized as follows.
 In \S 2, we make the deformation for the standard complex structure on torus $\mathbb{T}^4$.
 By this deformation, we give some interesting results on the standard torus $\mathbb{T}^4$.
 Finally in \S 3 we give the proof of our Main Theorem.

 \vskip 6pt

  \noindent{\bf Acknowledgements.}\,
   The first author would like to thank Professor Xiaojun Huang for his support.
  The first author also would like to thank Professors Tedi Draghici and Hongyu Wang for their patient guidance.
   The second author would like to thank East China Normal University and Professor Qing Zhou
   for hosting his visit in the spring semester in 2014.
   The authors dedicate this paper to the memory of Professor Ding Weiyue in deep appreciation for his long-term support of their work.
   The authors are grateful to the referees for their valuable comments and suggestions. Especially, the referees suggest a more direct
    argument about the openness part and point out the mistake in the proof of Lemma \ref{lemma a}.

 \section{Almost complex structures on $\mathbb{T}^4$}\setcounter{equation}{0}\label{specific example}

  In this section, we calculate the dimension of the $J$-anti-invariant cohomology
  subgroup on $4$-torus under the deformation of $\omega$-compatible almost complex structures.
  The following construction is a generalization of Example $2.6$ in \cite{TWZ} which is constructed by T. Draghici and C. H. Taubes (\cite{Dra,Taubes2}).
  These computations provide another example where the dimension of the $J$-anti-invariant cohomology is not an invariant
  under the deformation of almost complex structures.

  \begin{exam}\label{example}
  {\rm Let $\mathbb{T}^4$ be the standard torus with coordinates $\{x^1,x^2,x^3,x^4\}$.
  Denote by $(g_0,J_0,\omega_0)$ the standard flat K\"{a}hler structure on $\mathbb{T}^4$,
  where
  $$
  g_0=\sum_i dx^i\otimes dx^i,\,\,\, \omega_0=dx^1\wedge dx^2+dx^3\wedge dx^4.
  $$
  So $J_0$ is given by
  $$
  J_0dx^1=dx^2,\,\,\,J_0dx^2=-dx^1,\,\,\,J_0dx^3=dx^4,\,\,\,J_0dx^4=-dx^3.
  $$
  Equivalently, $J_0$ may be given by specifying
  $$
  \Lambda^-_{J_0}=Span\{dx^1\wedge dx^3-dx^2\wedge dx^4,\,\,\,dx^1\wedge dx^4+dx^2\wedge dx^3\}.
  $$
   It is well known that $b^+=b^-=3$.
   Indeed, we have
   $$
   \mathcal{H}^+_{g_0}=Span\{\omega_0, \,\,dx^1\wedge dx^3-dx^2\wedge dx^4,\,\,dx^1\wedge dx^4+dx^2\wedge dx^3 \}
   $$
   and
    $$
   \mathcal{H}^-_{g_0}=Span\{dx^1\wedge dx^2-dx^3\wedge dx^4, \,\,dx^1\wedge dx^3+dx^2\wedge dx^4,
    \,\,dx^1\wedge dx^4-dx^2\wedge dx^3 \}.
   $$
   Further more, by a direct calculation, we can get
   $$
   H^+_{J_0}=Span\{[\omega_0],[dx^1\wedge dx^2-dx^3\wedge dx^4],[dx^1\wedge dx^3+dx^2\wedge dx^4],
  [dx^1\wedge dx^4-dx^2\wedge dx^3] \}
   $$
   and
   $$
    H^-_{J_0}=Span\{[dx^1\wedge dx^3-dx^2\wedge dx^4],\,\,[dx^1\wedge dx^4+dx^2\wedge dx^3] \}.
   $$
   Hence, $h^+_{J_0}=4$ and $h^-_{J_0}=2$.

  Suppose $\Psi$ is a linear symplectomorphism.
  It is well known that if $\lambda$ is a eigenvalue of $\Psi$ then $\frac{1}{\lambda}$ occurs (cf. \cite{MS2}).
   According to such a fact, we can construct the almost complex structures as follows.
  Let
  $$A=e^{\sin 2\pi(x^1+x^3)},\,\,\,B=e^{\sin 2\pi(x^1+x^4)},\,\,\,C=
  e^{\frac{1}{2}[\sin 2\pi(x^1+x^3)-\sin 2\pi(x^1+x^4)]}$$
  and
  $$D=e^{-\frac{1}{2}[\sin 2\pi(x^1+x^3)+\sin 2\pi(x^1+x^4)]}.$$
  Obviously, $C^2=\frac{A}{B}$ and $D^2=\frac{1}{AB}$.
  Consider the almost complex structure $J$ given by
  $$
  Jdx^1=Cdx^2,\,\,\,Jdx^2=-\frac{1}{C}dx^1,\,\,\,Jdx^3=Ddx^4,\,\,\,Jdx^4=-\frac{1}{D}dx^3,
  $$
  It is easy to see that
  \begin{eqnarray*}
    \Lambda^+_J &=& Span\{dx^1\wedge dx^2+dx^3\wedge dx^4,dx^1\wedge dx^2-dx^3\wedge dx^4, \\
                & & Bdx^1\wedge dx^3+dx^2\wedge dx^4,dx^1\wedge dx^4-Adx^2\wedge dx^3\},
  \end{eqnarray*}

  $$
  \Lambda^-_J=Span\{dx^1\wedge dx^4+Adx^2\wedge dx^3,Bdx^1\wedge dx^3-dx^2\wedge dx^4\},
  $$
  and $J$ is compatible with $\omega_0$.
  The corresponding metric is
  $$g(\cdot,\cdot)=\omega_0(\cdot,J\cdot)=\frac{1}{C}dx^1\otimes dx^1+Cdx^2\otimes dx^2+
  \frac{1}{D}dx^3\otimes dx^3+Ddx^4\otimes dx^4.$$
  Direct calculation shows
  \begin{eqnarray*}
    \mathcal{H}^+_g &=& Span\{\omega_0,\,\,\,(\frac{1}{1+A}-c_A)\omega_0+\frac{1}{1+A}(dx^1\wedge
     dx^4+Adx^2\wedge dx^3),  \\
     & & (\frac{1}{1+B}-c_B)\omega_0+\frac{1}{1+B}(Bdx^1\wedge dx^3-dx^2\wedge dx^4)\},
  \end{eqnarray*}
  where $$ c_A=\int_{\mathbb{T}^4} \frac{1}{1+A}dvol \,\,\, {\rm and} \,\,\,
   c_B=\int_{\mathbb{T}^4} \frac{1}{1+B}dvol.$$
  Obviously,
  $$[(\frac{1}{1+A}-c_A)\omega_0+\frac{1}{1+A}(dx^1\wedge dx^4+Adx^2\wedge dx^3)]\wedge\omega_0
    =(\frac{1}{1+A}-c_A)\omega^2_0 \not\equiv 0$$
  and
  $$[(\frac{1}{1+B}-c_B)\omega_0+\frac{1}{1+B}(Bdx^1\wedge dx^3-dx^2\wedge dx^4)]\wedge\omega_0
    =(\frac{1}{1+B}-c_B)\omega^2_0 \not\equiv 0.$$
    Note that $\mathcal{Z}^-_J\subset \mathcal {H}_g^+$
    and for any $\alpha\in\mathcal{Z}^-_J$,
     we must have $\alpha\wedge\omega_0\equiv 0$,
    so we can obtain $\mathcal{Z}^-_J=\{0\}$ which
  implies that $h^-_J=0$ and $h^+_J=6$.
  By the definition of $H_J^{-,\perp}$ and Remark \ref{g rem},
  we know that both
  $$[(\frac{1}{1+A}-c_A)\omega_0+\frac{1}{1+A}(dx^1\wedge dx^4+Adx^2\wedge dx^3)]
   $$
  and
  $$[(\frac{1}{1+B}-c_B)\omega_0+\frac{1}{1+B}(Bdx^1\wedge dx^3-dx^2\wedge dx^4)]
    $$
  are in $H_J^{-,\perp}$.

   \vskip 6pt

   Denote by
   $e^i\triangleq dx^i$ and $e^{ij}\triangleq dx^i\wedge dx^j$.
   Let $$\omega_0=e^{12}+e^{34}\,, \quad\quad\alpha_0=e^{12}-e^{34},$$
  $$\omega_1=(\frac{1}{1+A}-c_A)\omega_0+\frac{1}{1+A}(e^{14}+Ae^{23}),$$
  $$\alpha_1=(\frac{1}{1+A}-c_A)\alpha_0+\frac{1}{1+A}(e^{14}-Ae^{23}),$$
  $$\omega_2=(\frac{1}{1+B}-c_B)\omega_0+\frac{1}{1+B}(Be^{13}-e^{24}),$$
  and
  $$\alpha_2=(\frac{1}{1+B}-c_B)\alpha_0+\frac{1}{1+B}(Be^{13}+e^{24}).$$
  Consider the element $e^{13}-e^{24}$ which is $J_0$-anti-invariant.
  Since $$(e^{13}-e^{24})\wedge\omega_0=0,\,\,\,
  (e^{13}-e^{24})\wedge\omega_1=0$$
  and
  $$\int_{\mathbb{T}^4}(e^{13}-e^{24})\wedge\omega_2=\int_{\mathbb{T}^4}e^{1234}=1,$$
   by Hodge decomposition (cf. pp. 10 in \cite{DK}),
  we can get
  \begin{equation}\label{projection 1}
     P^+_g(e^{13}-e^{24})=\frac{1}{\parallel\omega_2\parallel^2_{L^2(\mathbb{T}^4,g)}}\omega_2+d^+_g\gamma_1.
  \end{equation}
  On the other hand,
  $$(e^{13}-e^{24})\wedge\alpha_0=0,\,\,\,
  (e^{13}-e^{24})\wedge\alpha_1=0$$
  and
  $$\int_{\mathbb{T}^4}(e^{13}-e^{24})\wedge\alpha_2
  =\int_{\mathbb{T}^4}\frac{B-1}{B+1}e^{1234}\triangleq a,$$
  so we have
  \begin{equation}\label{projection 2}
    P^-_g(e^{13}-e^{24})=\frac{a}{\parallel\alpha_2\parallel^2_{L^2(\mathbb{T}^4,g)}}\alpha_2+d^-_g\gamma_2.
  \end{equation}
  Here $d^+_g\triangleq P^+_g\circ d$, $d^-_g\triangleq P^-_g\circ d$ and
   $\gamma_1, \gamma_2\in\Omega^1$.
   By (\ref{projection 1}) and (\ref{projection 2}),
   we can get the following equation
   \begin{equation}\label{13-24}
   e^{13}-e^{24}=\frac{1}{\parallel\omega_2\parallel^2_{L^2(\mathbb{T}^4,g)}}\omega_2
    +\frac{a}{\parallel\alpha_2\parallel^2_{L^2(\mathbb{T}^4,g)}}\alpha_2+d^+_g\gamma_1+d^-_g\gamma_2.
   \end{equation}
  Since $e^{13}-e^{24}$, $\omega_2$ and $\alpha_2$ are all closed,
  $d(d^+_g\gamma_1+d^-_g\gamma_2)=0$.
  Additionally,
  \begin{equation}\label{equ 0}
    0=(e^{13}-e^{24})\wedge\omega_0=\frac{1}{\parallel\omega_2\parallel^2_{L^2(\mathbb{T}^4,g)}}(\frac{1}{1+B}-c_B)\omega^2_0
  +d^+_g\gamma_1\wedge\omega_0,
  \end{equation}
  hence
  \begin{eqnarray}\label{equ +}
    d^+_g\gamma_1 &=& -\frac{1}{\parallel\omega_2\parallel^2_{L^2(\mathbb{T}^4,g)}}(\frac{1}{1+B}-c_B)\omega_0
            +P^-_J(d^+_g\gamma_1) \nonumber \\
        &=& -\frac{1}{\parallel\omega_2\parallel^2_{L^2(\mathbb{T}^4,g)}}(\frac{1}{1+B}-c_B)\omega_0
    +d^-_J\gamma_1.
  \end{eqnarray}
  Similarly, we have
   $$
     P^+_g(e^{14}+e^{23})=\frac{1}{\parallel\omega_1\parallel^2_{L^2(\mathbb{T}^4,g)}}\omega_1+d^+_g\theta_1
   $$
   and
  $$
    P^-_g(e^{14}+e^{23})=\frac{b}{\parallel\alpha_1\parallel^2_{L^2(\mathbb{T}^4,g)}}\alpha_1+d^-_g\theta_2,
  $$
  where $\theta_1,\,\,\theta_2\in\Omega^1$ and $b\triangleq\int_{\mathbb{T}^4}\frac{A-1}{A+1}e^{1234}$.
  So
  \begin{equation}\label{14+23}
    e^{14}+e^{23}=\frac{1}{\parallel\omega_1\parallel^2_{L^2(\mathbb{T}^4,g)}}\omega_1
  +\frac{b}{\parallel\alpha_1\parallel^2_{L^2(\mathbb{T}^4,g)}}\alpha_1+d^+_g\theta_1+d^-_g\theta_2,
  \end{equation}
  and
  $d(d^+_g\theta_1+d^-_g\theta_2)=0$. } \hfill$\square$
  \end{exam}

  \vskip 6pt

  For a better understanding of the relationship between the cohomologies,
  please see the following table.

  \begin{center}
  \begin{tabular}{|c|l|}
    \hline

   $\mathcal{H}^+_{g_0}$ &  $\omega_0$, \,\,\, $e^{13}-e^{24}$, \,\,\, $e^{14}+e^{23}$ \\

   \hline

    $\mathcal{H}^-_{g_0}$ &  $e^{12}-e^{34}$, \,\,\, $e^{13}+e^{24}$, \,\,\, $e^{14}-e^{23}$  \\

   \hline

   $H_{J_0}^+$ &  $[\omega_0]$, \,\,\, $[e^{12}-e^{34}]$, \,\,\, $[e^{13}+e^{24}]$, \,\,\, $[e^{14}-e^{23}]$  \\

   \hline

   $H_{J_0}^-$ &  $[e^{13}-e^{24}]$, \,\,\, $[e^{14}+e^{23}]$    \\

   \hline

   $\mathcal{H}^+_g$ &  $\omega_0$, \,\,\, $(\frac{1}{1+A}-c_A)\omega_0+\frac{1}{1+A}(e^{14}+Ae^{23})$ \\
   & $(\frac{1}{1+B}-c_B)\omega_0+\frac{1}{1+B}(Be^{13}-e^{24})$   \\
   \hline
   $\mathcal{H}^-_g$ &  $e^{12}-e^{34}$, \,\,\,
    $(\frac{1}{1+A}-c_A)(e^{12}-e^{34})+\frac{1}{1+A}(e^{14}-Ae^{23})$,\\
    & $(\frac{1}{1+B}-c_B)(e^{12}-e^{34})+\frac{1}{1+B}(Be^{13}+e^{24})$   \\
   \hline

   $H^+_J$ &  $[\omega_0]$, \,\,\, $[e^{12}-e^{34}]$, \\
    & $[(\frac{1}{1+A}-c_A)(e^{12}-e^{34})+\frac{1}{1+A}(e^{14}-Ae^{23})]$,\\
    & $[(\frac{1}{1+B}-c_B)(e^{12}-e^{34})+\frac{1}{1+B}(Be^{13}+e^{24})]$, \\
    & $[(\frac{1}{1+A}-c_A)\omega_0+\frac{1}{1+A}(e^{14}+Ae^{23})]$, \\
    & $[(\frac{1}{1+B}-c_B)\omega_0+\frac{1}{1+B}(Be^{13}-e^{24})]$  \\

   \hline
   $H^-_J$ &  $0$              \\

   \hline

  \end{tabular}
  \end{center}
  \begin{center}
  {\bf Table 1.} Bases for $H_{J_0}^+$, $H_{J_0}^-$, $H_J^+$, etc.  of $\mathbb{T}^4$.
  \end{center}

  By the above example, we have the following theorem:

  \begin{theo}
   Let $\mathbb{T}^4$ be the standard torus
  and $(g_0,J_0,\omega_0)$ be a standard flat K\"{a}hler structure on $\mathbb{T}^4$.
  Set $$A_k=e^{\frac{1}{k}\sin 2\pi(x^1+x^3)},\,\,\, B_k=e^{\frac{1}{k}\sin 2\pi(x^1+x^4)},
   \,\,\, C_k=(\frac{A_k}{B_k})^{\frac{1}{2}}$$
  and $D_k=(\frac{1}{A_k B_k})^{\frac{1}{2}}$.
  Construct the almost K\"{a}hler structures $(g_k,J_k,\omega_0)$ by
  $$
  J_kdx^1=C_kdx^2,\,\,\,J_kdx^2=-\frac{1}{C_k}dx^1,\,\,\,J_kdx^3=D_kdx^4,\,\,\,J_kdx^4=-\frac{1}{D_k}dx^3,
  $$
  and $g_k(\cdot,\cdot)=\omega_0(\cdot,J_k\cdot)$.
  Then $J_k\rightarrow J_0$, $g_k\rightarrow g_0$ as $k\rightarrow \infty$.
  However, $h^-_{J_k}=0$ and $h^+_{J_k}=6$.
  \end{theo}

 \section{Proof of main result}\setcounter{equation}{0}

  In this section we prove the Main Theorem.
  Let us first describe the $C^\infty$-topology
  on the space $\mathcal {J}$ of $C^\infty$ almost complex structures on $M$.
  For $k=0,1,2,\cdots$, the space $\mathcal{J}^k$ of $C^k$ almost complex structures on $M$
  has a natural separable Banach manifold structure.
  The natural $C^\infty$-topology on $\mathcal{J}$ is induced by
  the sequence of $C^k$ semi-norms $\|\cdot\|_k$, $k=0,1,2,\cdots$.
  With this $C^\infty$-topology, $\mathcal{J}$ is a Fr\'{e}chet manifold.
  A complete metric which induces the $C^\infty$-topology on $\mathcal{J}$
  is defined by
  $$ d(J_1,J_2)=\sum_{k=0}^\infty\frac{\|J_1-J_2\|_k}{2^k(1+\|J_1-J_2\|_k)}. $$
  For details, see \cite{Au,DLZ2}.
  At first, we will prove the openness statement of Main Theorem.
   Please see the following proposition:
  \begin{prop}\label{open prop}
   Let $(M,\omega)$ be a closed symplectic $4$-manifold.
  Then the set of $\omega$-compatible almost complex structures $J$ on $M$ with $h^-_J=0$ is an open
  subset of $\mathcal{J}_{\omega}^{\rm c}$ in the $C^\infty$-topology.
  \end{prop}
  \begin{proof}
 In the proof of Theorem $1.1$ in \cite{TWZZ}, we have gotten that the set of almost complex structures $J$ on $M$ with $h_J^-=0$
 is an open subset of $\mathcal{J}$ in the $C^\infty$-topology.
 Note that $\mathcal{J}_{\omega}^{\rm c}$ is a subspace of $\mathcal{J}$ in the $C^\infty$-topology.
 Using the fact that the space of $\omega$-compatible $J$ inherits the subspace topology,
 we can get that the set of $\omega$-compatible almost complex structures $J$ on $M$ with $h^-_J=0$ is an open
  subset of $\mathcal{J}_{\omega}^{\rm c}$ in the $C^\infty$-topology.
 \end{proof}

  With the above proposition, we can get the openness statement in Main Theorem.
  It remains to prove the denseness statement of Main Theorem.
  In the following section, we will give the proof of the denseness statement.

  \vskip 6pt

  Suppose $(M,g,J,\omega)$ is a closed almost K\"{a}hler $4$-manifold.
  To prove the denseness statement,
  we may consider a family $J_t$ of almost complex structures on $M$ which is a deformation of $J$,
  that is, $J_t \rightarrow J$ in the $C^{\infty}$-topology as $t \rightarrow 0$.
  If $h_J^-=0$, then as noted in \cite{DLZ2},
   we can establish path-wise semi-continuity property for $h_J^-$
  which follows directly from Lemma \ref{Lejmi lemma}
  and a classical result of Kodaira and Morrow (\cite[Theorem 4.3]{KoMo})
  showing the upper semi-continuity of the kernel of a family of elliptic differential operators.
  Therefore $h_{J_t}^-=0$ for small $t$.

  \vskip 6pt

  We now assume that $m\triangleq h_J^- \ge 1$.
  Let $\alpha_1,\cdots,\alpha_m \in \mathcal Z_J^-$ be such that
  $\alpha_1,\cdots,\alpha_m$ is an orthonormal basis of $H^-_J(\cong\mathcal Z_J^-)$ with respect to the cup product.
  Clearly, $1\leq m\leq b^+-1$.
   Define $H_{J,0}^{-,\perp}\subsetneqq H_J^{-,\perp} $ to be
  $$H_{J,0}^{-,\perp}\triangleq\{\beta=f\omega+d^-_J\gamma\in H_J^{-,\perp} :\gamma\in\Omega^1,\int_Mfdvol_g=0 \}.$$
   Here $d^-_J\triangleq P^-_J\circ d$.
    Since $\omega\notin H_{J,0}^{-,\perp}$,
 then it is easy to get the following decompositions,
 \begin{eqnarray*}
   \mathcal{H}^+_g &=&\mathcal Z_J^-\oplus H_J^{-,\perp}  \\
    &=& \mathcal Z_J^-\oplus\mathbb{R}\cdot\omega\oplus H_{J,0}^{-,\perp}.
 \end{eqnarray*}

  By making the local deformation of $\omega$-compatible almost complex structures on $(M,\omega)$ due to gluing operation (\cite{Taubes,Wang}),
  we have the following proposition:
  \begin{prop}\label{dense prop}
  Let $(M,g,J,\omega)$ be a closed almost K\"{a}hler $4$-manifold with $h^-_J\geq 1$.
  There exists a sequence of $\omega$-compatible almost complex structures, $\{J_{k_m}\}$,
  on $M$ such that $J_{k_m}\rightarrow J$ as $k_m\rightarrow\infty$ and $h^-_{J_{k_m}}=0$.
  \end{prop}
  \begin{proof}
  Without loss of generality, we may assume that
  $1\leq h^-_J=m\leq b^+-1$, then let $\{\beta_j=f_j\omega+d^-_J\gamma_j\}$, $1\leq j\leq l(\triangleq b^+-1-m)$,
  be an orthonormal basis of $H_{J,0}^{-,\perp}$ with respect to the cup product.
  We have that $\int_Mf_jdvol=0$ and $f_j\not\equiv 0$.

  First step, we will construct a family of almost complex structures $\{J'_k\}$ which is a local deformation of $J$
  with $\dim H_{J'_k,0}^{-,\perp}\geq l$.
  If $h^-_J=m=b^+-1$, that is, $l=0$,
  then any deformation $\{J'_{k}\}$ of $J$ satisfies $\dim H_{J'_k,0}^{-,\perp}\geq l=0$.
  Hence, we just have to handle the case of $h^-_J<b^+-1$.

   If $1\leq h^-_J<b^+-1$, set
  \begin{align}
  S_J := \{ \beta \in H_{J,0}^{-,\perp} \mid \textstyle\int_M \beta^2 = 1 \}.
  \end{align}
  Then $S_J$ is a sphere of dimension $l-1$.
  Define a function $V:S_J\rightarrow \mathbb{R}$ as follows: for any
  $\beta=f\omega+d^-_J\gamma \in S_J$,
  \begin{align}
  V(\beta) := {\rm vol}\,(M\setminus f^{-1}(0)) = \textstyle\int_{M\setminus f^{-1}(0)}dvol_g.
  \end{align}
  Denote by
  \begin{align}\label{deltaJ}
  \mu_J\triangleq\inf_{\beta \in S_J} V(\beta).
  \end{align}
  By the result in \cite{TWZZ} (cf. \cite[formula (3.5)]{TWZZ}), we know that $\mu_J>0$.
  Let $M'\triangleq \bigcap^m_{i=1}(M\setminus\alpha^{-1}_i(0))$,
  $M'$ is an open dense set in $M$.
  Fix a point $p\in M'$,
  the fundamental theorem of Darboux \cite{Au} shows that there are a neighborhood $U_p$ of $p$ and diffeomorphism $\Phi$
 from $U_p$ onto $\Phi(U_p)\subset\mathbb{C}^2=\mathbb{R}^4$ such that $\omega|_{U_p}=\Phi^*\omega_0$, where $\Phi(p)=0\in\mathbb{C}^2$
 and
 \begin{eqnarray}\label{5eq1}
   \omega_0 &=& \frac{\sqrt{-1}}{2}(dz^1\wedge d\overline{z^1}+dz^2\wedge d\overline{z^2}). \nonumber
 \end{eqnarray}
  We can choose $U_p$ small enough
  such that ${\rm vol}(U_p)<\frac{1}{3}\mu_J$.
   Denote by $J_p$ the pull back of $J_0$ on $U_p$,
   that is, $J_p\triangleq\Phi^*J_0$,
   where $J_0$ is the standard complex structure on $\mathbb{C}^2$.
   At point $p$, we have $J(p)=J_p(p)$.
   Set $g_p(\cdot,\cdot)=\omega(\cdot,J_p\cdot)$ on $U_p$.
   On the other hand, we know that $g(\cdot,\cdot)=\omega(\cdot,J\cdot)$.
   So we can get
   $$
   g|_{U_p}=g_p|_{U_p}\cdot e^h,
   $$
   here $h$ is a symmetric $J$-anti-invariant $(2,0)$ tensor (cf. \cite{Kim}).
   Construct metric
   $$
   g'_k=g_p\cdot e^{(1-\varphi_k)h}
   $$
   using Darboux coordinate chart $\{x^1,x^2,x^3,x^4\}$ on $U_p$,
   where
   \begin{equation}
  \varphi_k(x)=\left\{
    \begin{array}{ll}
    1, & ~ |x|\leq\frac{1}{k} \\
      &   \\
    0, & ~ |x|\geq\frac{2}{k}.
   \end{array}
  \right.
  \end{equation}
  We can extend $g'_k$ to the whole manifold by $g'_k|_{|x|\leq\frac{1}{k}}=g_p$ and
  $g'_k|_{M\setminus\{|x|\leq\frac{2}{k}\}}=g$ by using gluing operation (cf. \cite{Taubes, Wang}).
  By $g'_k$ and $\omega$, we can get the unique almost complex structure $J'_k$ such that
  $(g'_k,J'_k,\omega)$ is an almost K\"{a}hler structure on $M$, and
  $g'_k\rightarrow g, J'_k\rightarrow J$ as $k\rightarrow\infty$.

  We know that if $h^-_J<b^+-1$, then $H_{J,0}^{-,\perp}\neq \{0\}$.
  Since
  $$
  \Lambda^+_{g'_k}=\mathbb{R}\omega\oplus\Lambda^-_{J'_k}\,\,\, {\rm and}
  \,\,\,\Lambda^+_{J'_k}=\mathbb{R}\omega\oplus\Lambda^-_{g'_k},
  $$
  and
  we have
    \begin{eqnarray}
      \beta_j &=& f_j\omega+d^-_J\gamma_j \nonumber\\
      &=&f_j\omega+P^-_{J'_k}d^-_J\gamma_j+P^+_{J'_k}d^-_J\gamma_j  \nonumber\\
       &=& f_j\omega+P^-_{J'_k}d^-_J\gamma_j+P^-_{g'_k}d^-_J\gamma_j.
    \end{eqnarray}
    We get $P^-_{g'_k}d^-_J\gamma_j|_{M\setminus\{|x|\leq\frac{2}{k}\}}\equiv 0$
    since $g'_k|_{M\setminus\{|x|\leq\frac{2}{k}\}}=g$ and $d^-_J\gamma_j|_{M\setminus\{|x|\leq\frac{2}{k}\}}\in \Lambda^+_{g'_k}$.
    Moreover, $P^-_{g'_k}d^-_J\gamma_j\not\equiv 0$ on $\{|x|\leq\frac{2}{k}\}$
    and $P^-_{g'_k}d^-_J\gamma_j|_{\{|x|=\frac{2}{k}\}}=0$.
     Let $D'=\{|x|\leq\frac{2}{k}\}$.
    By \cite{DK}, we know that
  $$
  P^-_{g'_k}d\delta_{g'_k}:\Omega^-_{g'_k}(D')\longrightarrow\Omega^-_{g'_k}(D')
  $$
  is a self--adjoint strongly elliptic operator.
   Hence we can solve the following Dirichlet problem of
   $g'_k$-anti-self-dual equations (cf. \cite{DK}):
   \begin{equation}\label{P^-_g 2}
   \left\{
  \begin{array}{ll}
  P^-_{g'_k}d\delta_{g'_k}\eta_{j,k}=P^-_{g'_k}d^-_J\gamma_j, & ~ {\rm on}\,\,\, D'  \\

  \eta_{j,k}|_{\partial D'}=0 .
  \end{array}
  \right.
  \end{equation}
   By the standard elliptic theory (cf. \cite{GT}),
  there exists a unique solution $\eta_{j,k}\in\Omega^-_{g'_k}(D')$ satisfying Equations (\ref{P^-_g 2}).
  Therefore,
  $$
  \beta_j=f_j\omega+P^-_{J'_k}d^-_J\gamma_j+d\delta_{g'_k}\eta_{j,k}-d^+_{g'_k}\delta_{g'_k}\eta_{j,k},
  $$
  $1\leq j\leq l$.
  Denote by $\tilde{\beta}_j\triangleq \beta_j-d\delta_{g'_k}\eta_{j,k}\in\mathcal{H}^+_{g'_k}(M)$,
 it is clear that
  $$\tilde{\beta}_j=f_j\omega+P^-_{J'_k}d^-_J\gamma_j-d^+_{g'_k}\delta_{g'_k}\eta_{j,k}\in\Omega^+_{g'_k}(M)$$
  and $d\tilde{\beta}_j=0$,
   then $\tilde{\beta}_j$ and $\beta_j$ are in the same cohomology class, i.e.,
   $[\tilde{\beta}_j]=[\beta_j]$.
  \begin{eqnarray}
    \tilde{\beta}_j|_{M\setminus \{|x|\leq\frac{2}{k}\}} &=& \beta_j|_{M\setminus \{|x|\leq\frac{2}{k}\}}  \nonumber\\
     &=& (f_j\omega+d^-_J\gamma_j)|_{M\setminus \{|x|\leq\frac{2}{k}\}}.
  \end{eqnarray}
  Since ${\rm vol}(M\setminus f^{-1}_j(0))\geq \mu_J$ and  ${\rm vol}(U_p)<\frac{1}{3}\mu_J$,
  we get, on $M\setminus U_p$,
  $$
  \tilde{\beta}_j\wedge\omega=\beta_j\wedge\omega=f_j\omega^2\not\equiv 0.
  $$
  Hence $\tilde{\beta}_j$ contains a non-trivial element $\tilde{\beta}_{j,k}=\tilde{f}_{j,k}\omega+P^-_{J'_k}d\tilde{\gamma}_{j,k}\in H_{J'_k,0}^{-,\perp}$,
  where $\int_M\tilde{f}_{j,k}dvol_{g'_k}=0$ and ${\rm vol}\,((M\setminus \tilde{f}_{j,k}^{-1}(0))\cap(M\setminus U_p))\geq\frac{2}{3}\mu_J$
  for $1\leq j\leq l$.
  So $\dim \mathcal{Z}^-_{J'_k}\leq m$, $\dim H_{J'_k,0}^{-,\perp}\geq l$ and
  $Span\{\tilde{\beta}_{1,k},\cdot\cdot\cdot,\tilde{\beta}_{l,k}\}\subseteq H_{J'_k,0}^{-,\perp}$.

   \vskip 6pt

  Second step, we will deform $J'_k$ to $J_k$ satisfied $\dim H_{J_k,0}^{-,\perp}\geq l+1$ on $U_p$.
  The deformation is constructed similarly as the one in Example \ref{example}.
   Then by using gluing operation (cf. \cite{Taubes, Wang}),
  we have the following lemma:

  \begin{lem}\label{lemma a}
  There exists a sequence of $\omega$-compatible almost complex structures
  $\{J_k\}$ such that $J_k\rightarrow J$ as $k\rightarrow \infty$ and $\dim H_{J,0}^{-,\perp}\leq\dim H_{J'_k,0}^{-,\perp}\leq\dim H_{J_k,0}^{-,\perp}-1$,
  that is, $h^-_J-1\geq h^-_{J'_k}-1\geq h^-_{J_k}$.
  \end{lem}

  The above lemma will be proved later.
  \begin{rem}
  In the process of proving {\rm\cite[Theorem 1.1]{TWZZ}} which can be considered as the taming case,
  the following proposition play a key role.
  Suppose that $(M,g,J,\omega)$ is a closed almost K\"{a}hler $4$-manifold.
  If $J'$ is a $g$-related almost complex structure on $M$ with $J'\neq \pm J$,
  then $\dim(H^-_J\cap H^-_{J'})\leq 1$ ({\rm cf. \cite[Proposition 3.7]{DLZ2}}).
  But here, the compatible case, the above proposition does not work.
  Fortunately, {\rm Lemma \ref{lemma a}} will play a key role in the proof of our {\rm Main Theorem}.
  \end{rem}

  Now, let us
  return to the proof of Proposition \ref{dense prop}.
  With Lemma \ref{lemma a},
  we denote by $J_{k_1}\triangleq J_k$.
  Similar to the above discussion,
  after finite steps (at most for $m$ steps),
  we can construct almost complex structures $J_{k_m}$ on $M$
  such that $J_{k_m}\rightarrow J$ in $C^\infty$-topology as $k_m\rightarrow\infty$
  and $\dim H_{J_{k_m},0}^{-,\perp}=b^+-1$, that is,
  $h^-_{J_{k_m}}=0$.
  This completes the proof of Proposition \ref{dense prop}.
   \end{proof}

   From Proposition \ref{open prop} and \ref{dense prop},
   it is easy to get Main Theorem.

 \vskip 12pt

  In the remainder section, we will give the proof of Lemma \ref{lemma a}.

   \vskip 6pt

   \noindent {\bf Proof of Lemma \ref{lemma a}.}
   We have obtained $h^-_{J'_k}\leq h^-_J=m$ and $\dim H_{J'_k,0}^{-,\perp}\geq l$.
    If $\dim H_{J'_k,0}^{-,\perp}>l$, Lemma \ref{lemma a} holds automatically.
  Hence, we just have to prove the case of $\dim H_{J'_k,0}^{-,\perp}=l$.
  Let $\alpha'_1,\cdots,\alpha'_m \in \mathcal{Z}^-_{J'_k}$ be such that
  $\alpha'_1,\cdots,\alpha'_m$ is an orthonormal basis of $\mathcal{Z}^-_{J'_k}$ with respect to the cup product.
  Suppose $\{\beta'_1,\cdots,\beta'_l\}$
  is an orthonormal basis of $H_{J'_k,0}^{-,\perp}$ with respect to the cup product.

  The main idea of the following part is that for each $k\in\mathbb{N}$,
  we will construct a $1$-parameter family of almost complex structures $\{J^{\lambda}_k\}$, $\lambda\in[1,+\infty)$,
  which is a deformation of $J$ with $\dim H_{J^{\lambda}_k,0}^{-,\perp}\geq l+1$.
  The concrete construction of $J^{\lambda}_k$ is provided by rotating the Darboux coordinates and reforming the almost complex structure $J'_k$.

  First step, fixed $k\in\mathbb{N}$, we will construct the specific deformation $\{J^{\lambda}_k\}$, $\lambda\in[1,+\infty)$.

  Note that $H_{J'_k,0}^{-,\perp}$ is also spanned by
  $\{\tilde{\beta}_{1,k},\cdot\cdot\cdot,\tilde{\beta}_{l,k}\}$.
  Hence,
    $\mu_{J'_k}\geq\frac{2}{3}\mu_J$.
  Since $J'_k|_{|x|\leq\frac{1}{k}}=J_p(\triangleq\Phi^*J_0)$,
  it is easy to see that
  $\Lambda^-_{J'_k}$ is spanned by $\{dx^1\wedge dx^3-dx^2\wedge dx^4,dx^1\wedge dx^4+dx^2\wedge dx^3\}$ on $\{|x|\leq\frac{1}{k}\}$.
  Then $\alpha'_m|_{|x|\leq\frac{1}{k}}$ can be written as
  $$\alpha'_m|_{|x|\leq\frac{1}{k}}=L_1(x)(dx^1\wedge dx^3-dx^2\wedge dx^4)+L_2(x)(dx^1\wedge dx^4+dx^2\wedge dx^3),$$
  where $L_1(x)$ and $L_2(x)$ are smooth functions on $\{|x|\leq\frac{1}{k}\}$.
  By C. B\"{a}r's result (cf. \cite{Ba}),
  the set $\alpha'^{-1}_m(0)$ has Hausdorff dimension $\leq 2$.
  So we can choose a small open set $V\subseteq\{|x|\leq\frac{1}{k}\}$ such that
  $L^2_1(x)+L^2_2(x)|_V>0$.
  Without loss of generality,
  we assume $V=\{|x|\leq\frac{1}{k}\}$ and $L_1\not\equiv 0$ on $V$.
  We make a rotation for the Darboux coordinates $\{x^1,x^2,x^3,x^4\}$ such that
  $$
  dx^1=d\xi^1\cos\theta_1-d\xi^2\sin\theta_1,\,\,\, dx^2=d\xi^1\sin\theta_1+d\xi^2\cos\theta_1,
  $$
  $$
  dx^3=d\xi^3\cos\theta_2-d\xi^4\sin\theta_2, \,\,\, dx^4=d\xi^3\sin\theta_2+d\xi^4\cos\theta_2.
  $$
  We can find that
  $\omega|_{|x|\leq\frac{1}{k}}=d\xi^1\wedge d\xi^2+d\xi^3\wedge d\xi^4$
  and
  $|\xi|=|x|$.
  \begin{eqnarray*}
    \alpha'_m|_{|x|\leq\frac{1}{k}}
     &=& [L_1\cos(\theta_1+\theta_2)+L_2\sin(\theta_1+\theta_2)](d\xi^1\wedge d\xi^3-d\xi^2\wedge d\xi^4) \\
     &~& +[L_2\cos(\theta_1+\theta_2)-L_1\sin(\theta_1+\theta_2)](d\xi^1\wedge d\xi^4+d\xi^2\wedge d\xi^3).
  \end{eqnarray*}
 Choose some $\theta_1, \theta_2$ such that
  $L_2\cos(\theta_1+\theta_2)=L_1\sin(\theta_1+\theta_2)$.
  So
  $$\alpha'_m|_{|\xi|\leq\frac{1}{k}}=L(d\xi^1\wedge d\xi^3-d\xi^2\wedge d\xi^4),$$
  where $L\triangleq L_1\cos(\theta_1+\theta_2)+L_2\sin(\theta_1+\theta_2)$.
  Since $d\alpha'_m=0$,
   we obtain that $L$ is a nonzero constant on $\{{|\xi|\leq\frac{1}{k}}\}$.
   Moreover, we can assume that $L=1$.
   As in Section \ref{specific example}, for any $\lambda\in[1,+\infty)$, let
  $$
   A_{(k,\lambda)}(\xi)=e^{\phi_k(\xi)\sin 2\pi\lambda(\xi^1+\xi^3)},\,\,\,B_{(k,\lambda)}(\xi)=e^{\phi_k(\xi)\sin 2\pi\lambda(\xi^1+\xi^4)},
  $$
  $$C_{(k,\lambda)}(\xi)=e^{\frac{\phi_k(\xi)}{2}\sin 2\pi\lambda(\xi^1+\xi^3)-\frac{\phi_k(\xi)}{2}\sin 2\pi\lambda(\xi^1+\xi^4)},$$
  and
  $$D_{(k,\lambda)}(\xi)=e^{-\frac{\phi_k(\xi)}{2}\sin 2\pi\lambda(\xi^1+\xi^3)-\frac{\phi_k(\xi)}{2}\sin 2\pi\lambda(\xi^1+\xi^4)}.$$
  Here,
  \begin{equation}\label{cut-off}
  \phi_k(\xi)=\left\{
    \begin{array}{ll}
    1, & ~ |\xi|\leq\frac{1}{2k}, \\
      &   \\
    0, & ~ |\xi|\geq\frac{1}{k}.
   \end{array}
  \right.
  \end{equation}
  It is easy to see that
  $$
  A_{(k,\lambda)}(\xi)|_{|\xi|=\frac{1}{k}}=B_{(k,\lambda)}(\xi)|_{|\xi|=\frac{1}{k}}
  =C_{(k,\lambda)}(\xi)|_{|\xi|=\frac{1}{k}}=  D_{(k,\lambda)}(\xi)|_{|\xi|=\frac{1}{k}}=1.
  $$
  Define an almost complex structure $J^{\lambda}_k$ by gluing operation (cf. \cite{Taubes,Wang}) as follows:
  $ J^{\lambda}_k|_{|\xi|\geq\frac{1}{k}}=J'_k$
  and on ${|\xi|\leq\frac{1}{k}}$,
  \begin{equation*}
      J^{\lambda}_kd\xi^1=C_{(k,\lambda)}(\xi)d\xi^2,\,\,\,J^{\lambda}_kd\xi^2=-\frac{1}{C_{(k,\lambda)}(\xi)}d\xi^1,
  \end{equation*}
    \begin{equation}\label{J_k}
     J^{\lambda}_kd\xi^3=D_{(k,\lambda)}(\xi)d\xi^4,\,\,\,J^{\lambda}_kd\xi^4=-\frac{1}{D_{(k,\lambda)}(\xi)}d\xi^3.
   \end{equation}
   It is easy to see that $J^{\lambda}_k$ is compatible with $\omega$
   and $J^{\lambda}_k|_{|\xi|=\frac{1}{k}}=J_0$,
  where $J_0$ is the standard complex structure on $\mathbb{R}^4\cong \mathbb{C}^2$.
  Moreover, on ${|\xi|\leq\frac{1}{k}}$, we have
  $$\Lambda^-_{J^{\lambda}_k}=Span\{B_{(k,\lambda)}(\xi)d\xi^1\wedge d\xi^3-d\xi^2\wedge d\xi^4,\,\,\,d\xi^1\wedge d\xi^4+A_{(k,\lambda)}(\xi)d\xi^2\wedge d\xi^3\}.$$
  Set
  \begin{equation}\label{g_k}
    g^{\lambda}_k\triangleq\omega(\cdot,J^{\lambda}_k\cdot)=\frac{1}{C_{(k,\lambda)}}d\xi^1\otimes d\xi^1+C_{(k,\lambda)}d\xi^2\otimes d\xi^2
    +\frac{1}{D_{(k,\lambda)}}d\xi^3\otimes d\xi^3+D_{(k,\lambda)}d\xi^4\otimes d\xi^4.
  \end{equation}
 Then $g^{\lambda}_k\rightarrow g, J^{\lambda}_k\rightarrow J$ as $k\rightarrow\infty$,
 where $J$ is the almost complex structure defined on the closed symplectic $4$-manifold $(M,\omega)$ in Lemma \ref{lemma a}.
 Restricted on ${|\xi|\leq\frac{1}{2k}}$,
    $$
   A_{(k,\lambda)}(\xi)= A_{\lambda}(\xi)=e^{\sin 2\pi\lambda(\xi^1+\xi^3)},\,\,\,B_{(k,\lambda)}(\xi)=B_{\lambda}(\xi)=e^{\sin 2\pi\lambda(\xi^1+\xi^4)},
  $$
  $$C_{(k,\lambda)}(\xi)=C_{\lambda}(\xi)=e^{\frac{1}{2}\sin 2\pi\lambda(\xi^1+\xi^3)-\frac{1}{2}\sin 2\pi\lambda(\xi^1+\xi^4)}$$
  and
  $$D_{(k,\lambda)}(\xi)=D_{\lambda}(\xi)=e^{-\frac{1}{2}\sin 2\pi\lambda(\xi^1+\xi^3)-\frac{1}{2}\sin 2\pi\lambda(\xi^1+\xi^4)}.$$
 It is easy to see that $A_{\lambda}(\xi), B_{\lambda}(\xi),C_{\lambda}(\xi)$ and $D_{\lambda}(\xi)$
   are periodic functions defined on $\mathbb{R}^4$ with period $1/\lambda$, where $\lambda\in[1,+\infty)$.
   Hence, as in Example \ref{example}, for any $\lambda\in[1,+\infty)$,
   we can defined an almost complex structure $J^{\lambda}$ on torus $\mathbb{T}^4$ as follows:
    \begin{equation*}
      J^{\lambda}d\xi^1=C_{\lambda}(\xi)d\xi^2,\,\,\,J^{\lambda}d\xi^2=-\frac{1}{C_{\lambda}(\xi)}d\xi^1,
  \end{equation*}
    \begin{equation}
     J^{\lambda}d\xi^3=D_{\lambda}(\xi)d\xi^4,\,\,\,J^{\lambda}d\xi^4=-\frac{1}{D_{\lambda}(\xi)}d\xi^3.
   \end{equation}
     The corresponding compatible metric is
     \begin{equation}
    g^{\lambda}\triangleq\omega_0(\cdot,J^{\lambda}\cdot)=\frac{1}{C_{\lambda}}d\xi^1\otimes d\xi^1+C_{\lambda}d\xi^2\otimes d\xi^2
    +\frac{1}{D_{\lambda}}d\xi^3\otimes d\xi^3+D_{\lambda}d\xi^4\otimes d\xi^4.
  \end{equation}
    Hence, $H^-_{J^\lambda}=\{0\}$ and
     $$
     H^+_{J^\lambda}=Span\{\omega_0,\,\,\,\alpha_0,\,\,\,\omega^\lambda_1,\,\,\,\omega^\lambda_2,\,\,\,\alpha^\lambda_1,\,\,\,\alpha^\lambda_2\},
      $$
    where $\omega_0=d\xi^1\wedge d\xi^2+d\xi^3\wedge d\xi^4$, $\alpha_0=d\xi^1\wedge d\xi^2-d\xi^3\wedge d\xi^4$,
      $$\omega^\lambda_1=(\frac{1}{1+A_{\lambda}(\xi)}-c_{A_{\lambda}})\omega_0+\frac{1}{1+A_{\lambda}(\xi)}(d\xi^1\wedge d\xi^4+A_{\lambda}(\xi)d\xi^2\wedge d\xi^3),$$
    $$\omega^\lambda_2=(\frac{1}{1+B_{\lambda}(\xi)}-c_{B_{\lambda}})\omega_0+\frac{1}{1+B_{\lambda}(\xi)}(B_{\lambda}(\xi)d\xi^1\wedge d\xi^3-d\xi^2\wedge d\xi^4),$$
    $$\alpha^\lambda_1=(\frac{1}{1+A_{\lambda}(\xi)}-c_{A_{\lambda}})\alpha_0+\frac{1}{1+A_{\lambda}(\xi)}(d\xi^1\wedge d\xi^4-A_{\lambda}(\xi)d\xi^2\wedge d\xi^3),$$
   $$\alpha^\lambda_2=(\frac{1}{1+B_{\lambda}(\xi)}-c_{B_{\lambda}})\alpha_0+\frac{1}{1+B_{\lambda}(\xi)}(B_{\lambda}(\xi)d\xi^1\wedge d\xi^3+d\xi^2\wedge d\xi^4).$$
   Here $c_{A_{\lambda}}=\int_{\mathbb{T}^4}\frac{1}{1+A_{\lambda}(\xi)}dvol$  and $c_{B_{\lambda}}=\int_{\mathbb{T}^4}\frac{1}{1+B_{\lambda}(\xi)}dvol$.
   By a conformal transformation, $\lambda\xi\rightarrow x$, on $\mathbb{T}^4$, we have the following property,
   \begin{prop}
   $\omega_0(\xi)$, $\omega^\lambda_1(\xi)$, $\omega^\lambda_2(\xi)$, $\alpha_0(\xi)$, $\alpha^\lambda_1(\xi)$, and $\alpha^\lambda_2(\xi)$
   are respectively conformally equivalent to
    $\omega_0(x)$, $\omega_1(x)$, $\omega_2(x)$, $\alpha_0(x)$, $\alpha_1(x)$, and $\alpha_2(x)$ constructed in {\rm Example \ref{example}} for all $\lambda\in[1,+\infty)$.
   \end{prop}

 Second step, for any $\lambda\in[1,+\infty)$,
  we will find out part basis $\{\tilde{\beta'}^{\lambda}_{1,k},\cdots,\tilde{\beta'}^{\lambda}_{l,k}\}$ of $H_{J^{\lambda}_k,0}^{-,\perp}$.
  By the discussion in Proposition \ref{dense prop},
  $H_{J'_k,0}^{-,\perp}=Span\{\beta'_1,\cdots,\beta'_l\}$, where
  \begin{eqnarray}
      \beta'_j &=&f'_j\omega+d^-_{J'_k}\gamma'_j\nonumber\\
      &=&f'_j\omega+P^-_{J^{\lambda}_k}d^-_{J'_k}\gamma'_j+P^+_{J^{\lambda}_k}d^-_{J'_k}\gamma'_j  \nonumber\\
       &=& f'_j\omega+P^-_{J^{\lambda}_k}d^-_{J'_k}\gamma'_j+P^-_{g^{\lambda}_k}d^-_{J'_k}\gamma'_j.
    \end{eqnarray}
     Note that
  $$\Lambda^+_{g^{\lambda}_k}=\mathbb{R}\omega\oplus\Lambda^-_{J^{\lambda}_k}\,\,\, {\rm and}
  \,\,\,\Lambda^+_{J^{\lambda}_k}=\mathbb{R}\omega\oplus\Lambda^-_{g^{\lambda}_k}.
  $$
    Since $g^{\lambda}_k|_{M\setminus\{|\xi|\leq\frac{1}{k}\}}=g'_k$ and $J^{\lambda}_k|_{M\setminus\{|\xi|\leq\frac{1}{k}\}}=J'_k$,
    we can get $P^-_{g^{\lambda}_k}d^-_{J'_k}\gamma'_j|_{M\setminus\{|\xi|\leq\frac{1}{k}\}}\equiv 0$.
    Moreover, $P^-_{g^{\lambda}_k}d^-_{J'_k}\gamma'_j\not\equiv 0$ on $\{|\xi|\leq\frac{1}{k}\}$
    and $P^-_{g^{\lambda}_k}d^-_{J'_k}\gamma'_j|_{\{|\xi|=\frac{1}{k}\}}=0$.
     Let $D''=\{|\xi|\leq\frac{1}{k}\}$.
     Due to the standard elliptic theory (cf. \cite{GT}),
     there exists a unique solution ${\eta'}^{\lambda}_{j,k}\in\Omega^-_{g_k}(D'')$ for the following Dirichlet problem
    of $g^{\lambda}_k$-anti-self-dual equations (cf. \cite{DK}):
   \begin{equation}\label{P^-_g 3}
   \left\{
  \begin{array}{ll}
  P^-_{g^{\lambda}_k}d\delta_{g^{\lambda}_k}{\eta'}^{\lambda}_{j,k}=P^-_{g^{\lambda}_k}d^-_{J'_k}\gamma'_j, & ~ {\rm on}\,\,\, D''  \\

  {\eta'}^{\lambda}_{j,k}|_{\partial D''}=0 .
  \end{array}
  \right.
  \end{equation}
  Therefore,
  $$
  \beta'_j=f'_j\omega+P^-_{J^{\lambda}_k}d^-_{J'_k}\gamma'_j+d\delta_{g^{\lambda}_k}{\eta'}^{\lambda}_{j,k}
  -d^+_{g^{\lambda}_k}\delta_{g^{\lambda}_k}{\eta'}^{\lambda}_{j,k},
  $$
  $1\leq j\leq l$.
  Denote by $\tilde{\beta'}^{\lambda}_j\triangleq \beta'_j-d\delta_{g^{\lambda}_k}{\eta'}^{\lambda}_{j,k}\in\mathcal{H}^+_{g^{\lambda}_k}$, then $[\tilde{\beta'}^{\lambda}_j]=[\beta'_j]$.
  \begin{eqnarray}
    \tilde{\beta'}^{\lambda}_j|_{M\setminus \{|\xi|\leq\frac{1}{k}\}} &=&
    (f'_j\omega+P^-_{J^{\lambda}_k}d^-_{J'_k}\gamma'_j-d^+_{g^{\lambda}_k}\delta_{g^{\lambda}_k}{\eta'}^{\lambda}_{j,k})|_{M\setminus \{|\xi|\leq\frac{1}{k}\}}\nonumber\\
    &=& \beta'_j|_{M\setminus \{|\xi|\leq\frac{1}{k}\}}  \nonumber\\
     &=& (f'_j\omega+d^-_{J'_k}\gamma'_j)|_{M\setminus \{|\xi|\leq\frac{1}{k}\}}.
  \end{eqnarray}
  So on $M\setminus \{|\xi|\leq\frac{1}{k}\}$,
  $$
  \tilde{\beta'}^{\lambda}_j\wedge\omega=\beta'_j\wedge\omega=f'_j\omega^2\not\equiv 0,
  $$
  since $f'_j|_{M\setminus \{|\xi|\leq\frac{1}{k}\}}\not\equiv 0$ by the construction above.
  Hence $\tilde{\beta'}^{\lambda}_j$ contains a non-trivial element
  $\tilde{\beta'}^{\lambda}_{j,k}\in H_{J^{\lambda}_k,0}^{-,\perp}\cap\mathcal{H}^+_{g^{\lambda}_k}(M)$.
  Note that on $M\setminus \{|\xi|\leq\frac{1}{k}\}$,
   $J^{\lambda}_k=J'_k$.
  So when restricted to $M\setminus \{|\xi|\leq\frac{1}{k}\}$,
  we will get
  $\tilde{\beta'}^{\lambda}_{j,k}=\tilde{\beta'}_j=\beta'_j$.
  This implies that $\{\tilde{\beta'}^{\lambda}_{1,k},\cdots,\tilde{\beta'}^{\lambda}_{l,k}\}$ are linearly independent.

   Third step, for some $\lambda\in[1,+\infty)$,
    we will construct another element $\tilde{\beta'}^{\lambda}_{l+1,k}$ in $H_{J^{\lambda}_k,0}^{-,\perp}\cap\mathcal{H}^+_{g^{\lambda}_k}(M)$ which is independent with $\{\tilde{\beta'}^{\lambda}_{1,k},\cdots,\tilde{\beta'}^{\lambda}_{l,k}\}$.
    By the discussion in the first step, $\alpha'_m\in H^-_{J'_k}(M)$ and
  $\alpha'_m|_{|\xi|\leq\frac{1}{k}}=d\xi^1\wedge d\xi^3-d\xi^2\wedge d\xi^4$.
    Hence, on $M\setminus \{|\xi|\leq\frac{1}{k}\}$,
        $P^+_{g^{\lambda}_k}\alpha'_m=\alpha'_m$ since $g^{\lambda}_k|_{M\setminus\{|\xi|\leq\frac{1}{k}\}}=g'_k|_{M\setminus\{|\xi|\leq\frac{1}{k}\}}$.
    On $\{|\xi|\leq\frac{1}{k}\}$,
      \begin{eqnarray}\label{proj1}
        P^+_{g^{\lambda}_k}\alpha'_m &=& \frac{1}{2}\{\alpha'_m-J^{\lambda}_k\alpha'_m\} \nonumber\\
         &=& \frac{1}{2}\{(1+B_{(k,\lambda)}(\xi))d\xi^1\wedge d\xi^3-(1+\frac{1}{B_{(k,\lambda)}(\xi)})d\xi^2\wedge d\xi^4\}.
      \end{eqnarray}
  Similarly, we have
  \begin{equation}
    P^-_{g^{\lambda}_k}\alpha'_m|_{M\setminus\{|\xi|\leq\frac{1}{k}\}}\equiv 0
  \end{equation}
   and
     \begin{equation}
    P^-_{g^{\lambda}_k}\alpha'_m|_{|\xi|\leq\frac{1}{k}}
    =\frac{1}{2}\{(1-B_{(k,\lambda)}(\xi))d\xi^1\wedge d\xi^3-(1-\frac{1}{B_{(k,\lambda)}(\xi)})d\xi^2\wedge d\xi^4\}
  \end{equation}
    By Hodge theory (cf. \cite{DK}),
    there exists a unique $\zeta_{(k,\lambda)}(\xi)\in\Omega^-_{g^{\lambda}_k}(\{|\xi|\leq\frac{1}{k}\})$
    satisfying the Dirichlet problem of $g^{\lambda}_k$-anti-self-dual equations:
    \begin{equation}\label{proj4}
   \left\{
  \begin{array}{ll}
  P^-_{g^{\lambda}_k}d\delta_{g^{\lambda}_k}\zeta_{(k,\lambda)}(\xi)=\frac{1}{2}\{(1-B_{(k,\lambda)}(\xi))d\xi^1\wedge d\xi^3-(1-\frac{1}{B_{(k,\lambda)}(\xi)})d\xi^2\wedge d\xi^4\}  \\

  \zeta_{(k,\lambda)}(\xi)|_{|\xi|=\frac{1}{k}}=0 .
  \end{array}
  \right.
  \end{equation}
   Hence, let $\beta_{(k,\lambda)}(\xi)=\delta_{g^{\lambda}_k}\zeta_{(k,\lambda)}(\xi)$, by (\ref{proj1})-(\ref{proj4}),
   we have \begin{eqnarray}
             \tilde{\beta'}^{\lambda}_{l+1} &\triangleq& \alpha'_m-d\beta_{(k,\lambda)}(\xi) \nonumber\\
              &=&    P^+_{g^{\lambda}_k}\alpha'_m-d^+_{g^{\lambda}_k}\beta_{(k,\lambda)}(\xi)\in\Omega^+_{g^{\lambda}_k}(M)
           \end{eqnarray}
    is a $d$-closed $g^{\lambda}_k$-self-dual $2$-form which is cohomologous to $\alpha'_m$, that is, $[\tilde{\beta'}^{\lambda}_{l+1}]=[\alpha'_m]$.
    Therefore, $\tilde{\beta'}^{\lambda}_{l+1}\in\mathcal{H}^+_{g^{\lambda}_k}(M)$ for any $\lambda\in[1,+\infty)$.
     By (\ref{proj1}), it is easy to see that $P^+_{g^{\lambda}_k}\alpha'_m\wedge\omega\equiv0$.
     Note that $\omega|_{|\xi|\leq\frac{1}{k}}=\omega_0$ and $d^+_{g^{\lambda}_k}\beta_{(k,\lambda)}|_{M\setminus\{|\xi|\leq\frac{1}{k}\}}\equiv 0$.
     If $d^+_{g^{\lambda}_k}\beta_{(k,\lambda)}\wedge\omega_0\not\equiv 0$ on $\{|\xi|\leq\frac{1}{k}\}$,
     then $\tilde{\beta'}^{\lambda}_{l+1}$ contains an element in $H^{-,\perp}_{J^{\lambda}_k,0}$.
     To complete the proof of Lemma \ref{lemma a}, we need the following claim,
    \begin{cla}\label{claim}
    Fixed $k\in\mathbb{N}$, then there exists some $\lambda_0\in[1,+\infty)$ such that
    $$d^+_{g^{\lambda_0}_k}\beta_{(k,\lambda_0)}\wedge\omega_0\not\equiv 0$$ on $\{|\xi|\leq\frac{1}{k}\}$.
    \end{cla}

    The above claim will be proved later.

   Now, let us return to the proof of Lemma \ref{lemma a}.
   With Claim \ref{claim}, we can find a $\lambda_0\in[1,+\infty)$ such that $d^+_{g^{\lambda_0}_k}\beta_{(k,\lambda_0)}\wedge\omega_0\not\equiv 0$.
   Then there is a point $q_1\in\{|\xi|\leq\frac{1}{k}\}$ such that $\tilde{\beta'}^{\lambda_0}_{l+1}\wedge\omega_0|_{q_1}\neq 0$.
   It implies that $\tilde{\beta'}^{\lambda_0}_{l+1}\wedge\omega\neq 0$ on a small neighborhood of $q_1$ in $\{|\xi|\leq\frac{1}{k}\}$.
  Thus, $\tilde{\beta'}^{\lambda_0}_{l+1}$ contains a non-trivial element $\tilde{\beta'}^{\lambda_0}_{l+1,k}$ in $H^{-,\perp}_{J^{\lambda_0}_k,0}$.
   Note that $\tilde{\beta'}^{\lambda_0}_{l+1}=\tilde{f}'_{l+1}\omega+d^-_{J^{\lambda_0}_k}\tilde{\gamma}'_{l+1}$,
   where ${\rm supp}\tilde{f}'_{l+1}\subset U_p$ (defined in the first step of the proof of Proposition \ref{dense prop}) and ${\rm vol(supp}\tilde{f}'_{l+1})<\frac{\mu_J}{3}$.
   Then $\{\tilde{\beta'}^{\lambda_0}_{1,k},\cdots,\tilde{\beta'}^{\lambda_0}_{l,k},\tilde{\beta'}^{\lambda_0}_{l+1,k}\}$ are linearly independent.
    Let $J_k=J^{\lambda_0}_k$, $g_k=g^{\lambda_0}_k$, and $\tilde{\beta'}_{l,k}=\tilde{\beta'}^{\lambda_0}_{l,k}$, $1\leq j\leq l+1$.
    Therefor, $\{\tilde{\beta'}_{1,k},\cdots,\tilde{\beta'}_{l,k},\tilde{\beta'}_{l+1,k}\}$ is a part of $H^{-,\perp}_{J_k,0}$.
   This completes the proof of Lemma \ref{lemma a}.

  In the remainder section, we will give the proof of Claim \ref{claim}.

   \vskip 6pt

   \noindent {\bf Proof of Claim \ref{claim}.}
     Fix $k\in\mathbb{N}$ and suppose that
     \begin{equation}\label{supp1}
       d^+_{g^{\lambda}_k}\beta_{(k,\lambda)}\wedge\omega_0\equiv 0
     \end{equation}
      on $\{|\xi|\leq\frac{1}{k}\}$ for all $\lambda\in[1,+\infty)$.
    Here, $\beta_{(k,\lambda)}(\xi)=\delta_{g^{\lambda}_k}\zeta_{(k,\lambda)}(\xi)$
    satisfying Dirichlet problem for $g^{\lambda}_k$-anti-self-dual equations (cf. \cite{DK}):
    \begin{equation}\label{supp2}
   \left\{
  \begin{array}{ll}
  P^-_{g^{\lambda}_k}d\beta_{(k,\lambda)}(\xi)= P^-_{g^{\lambda}_k}\alpha'_m \,\,\,\,\,\,{\rm on}\,\,\,\{|\xi|<\frac{1}{k}\},\\

  \zeta_{(k,\lambda)}(\xi)|_{|\xi|=\frac{1}{k}}=0 .
  \end{array}
  \right.
  \end{equation}
  Using gluing operation (cf. \cite{Taubes,Wang}) and replacing ${M\setminus \{|\xi|\leq\frac{1}{k}\}}$ by ${\mathbb{R}^4\setminus \{|\xi|\leq\frac{1}{k}\}}$,
  define an almost complex structure $\tilde{J}^{\lambda}_k$ on $\mathbb{R}^4$ as follows
   $$
  \tilde{J}^{\lambda}_k|_{|\xi|\leq\frac{1}{k}}=J^{\lambda}_k,\,\,\, \tilde{J}^{\lambda}_k|_{\mathbb{R}^4\setminus \{|\xi|\leq\frac{1}{k}\}}=J_0,
   $$
   where $J^{\lambda}_k$ is defined by (\ref{J_k})
   and $J_0$ is the standard complex structure on $\mathbb{R}^4\cong\mathbb{C}^2$.
 Then, on $\mathbb{R}^4$, the almost complex structure $\tilde{J}^{\lambda}_k$ is given by
   \begin{equation*}
      \tilde{J}^{\lambda}_kd\xi^1=C_{(k,\lambda)}(\xi)d\xi^2,\,\,\,\tilde{J}^{\lambda}_kd\xi^2=-\frac{1}{C_{(k,\lambda)}(\xi)}d\xi^1,
  \end{equation*}
    \begin{equation}
     \tilde{J}^{\lambda}_kd\xi^3=D_{(k,\lambda)}(\xi)d\xi^4,\,\,\,\tilde{J}^{\lambda}_kd\xi^4=-\frac{1}{D_{(k,\lambda)}(\xi)}d\xi^3.
   \end{equation}
 Moreover, we can define the corresponding almost K\"{a}hler triples $(\tilde{J}^{\lambda}_k,\omega_0, \tilde{g}^{\lambda}_k)$ on $\mathbb{R}^4$,
   where $\omega_0=d\xi^1\wedge d\xi^2+d\xi^3\wedge d\xi^4$ and
   \begin{eqnarray}
     \tilde{g}^{\lambda}_k &\triangleq& \omega_0(\cdot,\tilde{J}^{\lambda}_k\cdot) \nonumber\\
      &=& \frac{1}{C_{(k,\lambda)}}d\xi^1\otimes d\xi^1+C_{(k,\lambda)}d\xi^2\otimes d\xi^2
    +\frac{1}{D_{(k,\lambda)}}d\xi^3\otimes d\xi^3+D_{(k,\lambda)}d\xi^4\otimes d\xi^4.\nonumber\\
    &~&
   \end{eqnarray}
   It is easy to see that
  \begin{eqnarray*}
    \Lambda^+_{\tilde{J}^{\lambda}_k} &=& Span\{d\xi^1\wedge d\xi^2+dx^3\wedge d\xi^4,d\xi^1\wedge d\xi^2-d\xi^3\wedge d\xi^4, \\
                & & B_{(k,\lambda)}d\xi^1\wedge d\xi^3+d\xi^2\wedge d\xi^4,d\xi^1\wedge d\xi^4-A_{(k,\lambda)}d\xi^2\wedge d\xi^3\}
  \end{eqnarray*}
  and
  $$
  \Lambda^-_{\tilde{J}^{\lambda}_k}=Span\{d\xi^1\wedge d\xi^4+A_{(k,\lambda)}d\xi^2\wedge d\xi^3,B_{(k,\lambda)}d\xi^1\wedge d\xi^3-d\xi^2\wedge d\xi^4\}.
  $$
  By direct calculation, we get
  \begin{eqnarray}
    P^+_{\tilde{g}^{\lambda}_k}(d\xi^1\wedge d\xi^3-d\xi^2\wedge d\xi^4) &=&\frac{1}{2}\{(1+B_{(k,\lambda)}(\xi))d\xi^1\wedge d\xi^3 \nonumber \\
     &~& -(1+\frac{1}{B_{(k,\lambda)}(\xi)})d\xi^2\wedge d\xi^4\}
  \end{eqnarray}
    and
   \begin{eqnarray}\label{pro g}
    P^-_{\tilde{g}^{\lambda}_k}(d\xi^1\wedge d\xi^3-d\xi^2\wedge d\xi^4)&=& \frac{1}{2}\{(1-B_{(k,\lambda)}(\xi))d\xi^1\wedge d\xi^3 \nonumber \\
    &~& -(1-\frac{1}{B_{(k,\lambda)}(\xi)})d\xi^2\wedge d\xi^4\}.
 \end{eqnarray}


  Make a conformal transformation on $\mathbb{R}^4$: $\xi\rightarrow\frac{x}{\lambda}$.
  Then $A_{(k,\lambda)}(\xi)$, $B_{(k,\lambda)}(\xi)$, $C_{(k,\lambda)}(\xi)$, $D_{(k,\lambda)}(\xi)$ transform correspondingly into
  $$
   A_{(k,\lambda)}(x)=e^{\phi_k(\frac{x}{\lambda})\sin 2\pi(x^1+x^3)},\,\,\,B_{(k,\lambda)}(x)=e^{\phi_k(\frac{x}{\lambda})\sin 2\pi(x^1+x^4)},
  $$
  $$C_{(k,\lambda)}(x)=e^{\frac{\phi_k(\frac{x}{\lambda})}{2}\sin 2\pi(x^1+x^3)-\frac{\phi_k(\frac{x}{\lambda})}{2}\sin 2\pi(x^1+x^4)},$$
  and
  $$D_{(k,\lambda)}(x)=e^{-\frac{\phi_k(\frac{x}{\lambda})}{2}\sin 2\pi(x^1+x^3)-\frac{\phi_k(\frac{x}{\lambda})}{2}\sin 2\pi(x^1+x^4)}.$$
  In particular,
    $$
   A_{(k,\lambda)}(x)|_{|x|\leq\frac{\lambda}{2k}}=e^{\sin 2\pi(x^1+x^3)}=A,\,\,\,A_{(k,\lambda)}(x)|_{\mathbb{R}^4\setminus \{|x|\leq\frac{\lambda}{k}\}}=1;
  $$
  $$
  B_{(k,\lambda)}(x)|_{|x|\leq\frac{\lambda}{2k}}=e^{\sin 2\pi(x^1+x^4)}=B,\,\,\,B_{(k,\lambda)}(x)|_{\mathbb{R}^4\setminus \{|x|\leq\frac{\lambda}{k}\}}=1;
  $$
  $$
  C_{(k,\lambda)}(x)|_{|x|\leq\frac{\lambda}{2k}}=e^{\frac{1}{2}\sin 2\pi(x^1+x^3)-\frac{1}{2}\sin 2\pi(x^1+x^4)}=C
  ,\,\,\,C_{(k,\lambda)}(x)|_{\mathbb{R}^4\setminus \{|x|\leq\frac{\lambda}{k}\}}=1;
  $$
  \begin{equation}\label{3.27}
     D_{(k,\lambda)}(x)|_{|x|\leq\frac{\lambda}{2k}}=e^{-\frac{1}{2}\sin 2\pi(x^1+x^3)-\frac{1}{2}\sin 2\pi(x^1+x^4)}=D,\,\,\,
  D_{(k,\lambda)}(x)|_{\mathbb{R}^4\setminus \{|x|\leq\frac{\lambda}{k}\}}=1.
  \end{equation}
  Here $A,B,C$ and $D$ are defined in Example \ref{example}.
  Therefore,
   $$
  \tilde{J}^{\lambda}_k|_{|x|\leq\frac{\lambda}{2k}}=J,\,\,\, \tilde{J}^{\lambda}_k|_{\mathbb{R}^4\setminus \{|x|\leq\frac{\lambda}{k}\}}=J_0,
   $$
   where $J$ is the almost complex structure on $\mathbb{R}^4$ defined in Example \ref{example}
   and $J_0$ is the standard complex structure
   on  $\mathbb{R}^4\cong\mathbb{C}^2$.
 In the new coordinate system,
  $\omega_0=\frac{1}{\lambda^2}(dx^1\wedge dx^2+dx^3\wedge dx^4)$ and
     \begin{equation}
    \tilde{g}^{\lambda}_k=\frac{1}{\lambda^2}(\frac{1}{C_{(k,\lambda)}}dx^1\otimes dx^1+C_{(k,\lambda)}dx^2\otimes dx^2
    +\frac{1}{D_{(k,\lambda)}}dx^3\otimes dx^3+D_{(k,\lambda)}dx^4\otimes dx^4).
  \end{equation}
 Define $g_{(k,\lambda)}\triangleq\lambda^2\tilde{g}^{\lambda}_k$ and
  still denote $dx^1\wedge dx^2+dx^3\wedge dx^4$ by $\omega_0$.
  The almost complex structure determined by $g_{(k,\lambda)}$ and $\omega_0$ is denoted by $J_{(k,\lambda)}$.
  Actually, $J_{(k,\lambda)}=\tilde{J}^{\lambda}_k$.
  Denote by $\Omega^-_{g_{(k,\lambda)}}(\mathbb{R}^4)_0$ the smooth space of $g_{(k,\lambda)}$-anti-self-dual $2$-forms on $\mathbb{R}^4$ with compact support.
    By the elliptic theory (cf. \cite{GT}),
     there exists a unique solution $\zeta_{(k,\lambda)}(x)\in\Omega^-_{g_{(k,\lambda)}}(\mathbb{R}^4)_0$ satisfying
     the following Dirichlet problem of $g_{(k,\lambda)}$-anti-self-dual equations (cf. \cite{DK}):
   \begin{equation}\label{g k}
   \left\{
  \begin{array}{ll}
  P^-_{g_{(k,\lambda)}}d\delta_{g_{(k,\lambda)}}\zeta_{(k,\lambda)}(x)
  =P^-_{g_{(k,\lambda)}}(dx^1\wedge dx^3-dx^2\wedge dx^4),  \\

  \zeta_{(k,\lambda)}(x)|_{|x|=\frac{\lambda}{k}}=0 .
  \end{array}
  \right.
  \end{equation}
    Then
  \begin{equation}
    P^+_{g_{(k,\lambda)}}(dx^1\wedge dx^3-dx^2\wedge dx^4)-d^+_{g_{(k,\lambda)}}\beta_{(k,\lambda)}
     =dx^1\wedge dx^3-dx^2\wedge dx^4-d\beta_{(k,\lambda)}
  \end{equation}
  is a $d$-closed $2$-form, where $\beta_{(k,\lambda)}=\delta_{g_{(k,\lambda)}}\zeta_{(k,\lambda)}$
  and $\beta_{(k,\lambda)}|_{\mathbb{R}^4\setminus \{|x|\leq\frac{\lambda}{k}\}}\equiv 0$.
    By (\ref{supp1}),
    \begin{equation}\label{supp6}
       d^+_{g_{(k,\lambda)}}\beta_{(k,\lambda)}\wedge\omega_0\equiv 0
     \end{equation}
      on $\mathbb{R}^4$ for all $\lambda\in[1,+\infty)$.

      As $\lambda\rightarrow+\infty$, by (\ref{3.27}),
    $A_{(k,\lambda)}(x)$, $B_{(k,\lambda)}(x)$, $C_{(k,\lambda)}(x)$, and $D_{(k,\lambda)}(x)$
    are converging to $A$, $B$, $C$, and $D$ (which are constructed in Example \ref{example}) respectively in the weak (compact-open)
   topology (cf. \cite[Chapter 2]{Hirsch}) .
   Hence, $J_{(k,\lambda)}\rightarrow J^{\infty}$, $g_{(k,\lambda)}\rightarrow g^{\infty}$,
   and $P^{\pm}_{g_{(k,\lambda)}}\rightarrow P^{\pm}_{g^{\infty}}$ in the weak (compact-open)
   topology as $\lambda\rightarrow+\infty$,
   where $J^{\infty}$ and $g^{\infty}$ are the same to the almost complex structure $J$ and the metric $g$ on $\mathbb{R}^4$ in Example \ref{example}.
   Hence, $J^{\infty}$ and $g^{\infty}$ are periodic structures on $\mathbb{R}^4$.

   By (\ref{projection 2}), if $P^-_{g^{\infty}}(dx^1\wedge dx^3-dx^2\wedge dx^4)$ is viewed as a $g^{\infty}$-anti-self-dual $2$-form on $\mathbb{T}^4$,
   then we have the following Hodge decomposition
   $P^-_{g^{\infty}}(dx^1\wedge dx^3-dx^2\wedge dx^4)=\frac{a}{\parallel\alpha_2\parallel^2_{L^2(\mathbb{T}^4,g)}}\alpha_2+d^-_{g^{\infty}}\gamma_2$,
   where $\alpha_2\in\mathcal{H}^-_{g^\infty}(\mathbb{T}^4)$ and $\gamma_2\in\Omega^1(\mathbb{T}^4)$.
   However, $\gamma_2$ can be viewed as a periodic $1$-form and $\alpha_2$ a periodic $g^{\infty}$-anti-self-dual $d$-closed $2$-form on $\mathbb{R}^4$.
   Since $\mathbb{R}^4$ is contractible, by Poincar\'{e} Lemma, there exists a $1$-form $\gamma_{\infty}$ on $\mathbb{R}^4$ such that
   \begin{equation}\label{contractible}
    \frac{a}{\parallel\alpha_2\parallel^2_{L^2(\mathbb{T}^4,g)}}\alpha_2=d\gamma_{\infty}=d^-_{g^{\infty}}\gamma_{\infty}.
   \end{equation}
   Thus, $P^-_{g^{\infty}}(dx^1\wedge dx^3-dx^2\wedge dx^4)=d^-_{g^{\infty}}(\gamma_{\infty}+\gamma_2)$.
    By (\ref{g k}),
    $d^-_{g_{(k,\lambda)}}\beta_{(k,\lambda)}=P^-_{g_{(k,\lambda)}}(dx^1\wedge dx^3-dx^2\wedge dx^4)$
    and $P^-_{g_{(k,\lambda)}}(dx^1\wedge dx^3-dx^2\wedge dx^4)|_{|x|\leq\frac{\lambda}{2k}}=P^-_{g^{\infty}}(dx^1\wedge dx^3-dx^2\wedge dx^4)$,
    so $d^-_{g_{(k,\lambda)}}\beta_{(k,\lambda)}|_{|x|\leq\frac{\lambda}{2k}}=d^-_{g^{\infty}}(\gamma_{\infty}+\gamma_2)$ for any $\lambda\in[1,+\infty)$.
    Then $d^-_{g_{(k,\lambda)}}\beta_{(k,\lambda)}\rightarrow d^-_{g^{\infty}}(\gamma_{\infty}+\gamma_2)$ as $\lambda\rightarrow +\infty$
    in weak topology on $\Omega^2(\mathbb{R}^4)$.
    By (\ref{contractible}),
    $d^+_{g_{(k,\lambda)}}\beta_{(k,\lambda)}\rightarrow d^+_{g^{\infty}}(\gamma_{\infty}+\gamma_2)=d^+_{g^{\infty}}\gamma_2$ as $\lambda\rightarrow +\infty$
    in weak topology on $\Omega^2(\mathbb{R}^4)$.
  Let $\beta_{\infty}\triangleq\gamma_{\infty}+\gamma_2$.
    By (\ref{supp6}), we get
    \begin{equation}\label{supp7}
       d^+_{g^{\infty}}\beta_{\infty}\wedge\omega_0=d^+_{g^{\infty}}\gamma_2\wedge\omega_0\equiv 0
     \end{equation}
      on $\mathbb{R}^4$.
 Note that
 $$ dx^1\wedge dx^3-dx^2\wedge dx^4 \qquad \qquad \qquad \qquad \qquad \qquad \qquad \qquad \qquad \qquad$$
 $$ \qquad \qquad=P^+_{g_{(k,\lambda)}}(dx^1\wedge dx^3-dx^2\wedge dx^4)+P^-_{g_{(k,\lambda)}}(dx^1\wedge dx^3-dx^2\wedge dx^4)$$
 $$ \qquad \qquad=P^+_{g^{\infty}}(dx^1\wedge dx^3-dx^2\wedge dx^4)+P^-_{g^{\infty}}(dx^1\wedge dx^3-dx^2\wedge dx^4). $$
 Then
 $$
 P^+_{g_{(k,\lambda)}}(dx^1\wedge dx^3-dx^2\wedge dx^4)\rightarrow P^+_{g^{\infty}}(dx^1\wedge dx^3-dx^2\wedge dx^4)
 $$
 and
 $$
 P^-_{g_{(k,\lambda)}}(dx^1\wedge dx^3-dx^2\wedge dx^4)\rightarrow P^-_{g^{\infty}}(dx^1\wedge dx^3-dx^2\wedge dx^4)
 $$
 as $\lambda\rightarrow +\infty$ in weak topology .
  Since $P^+_{g^{\infty}}(dx^1\wedge dx^3-dx^2\wedge dx^4)$ can be regarded as $g^{\infty}$-sefl-dual $2$-form on $\mathbb{T}^4$,
   by (\ref{projection 1}),
    $P^+_{g^{\infty}}(dx^1\wedge dx^3-dx^2\wedge dx^4)=\frac{1}{\parallel\omega_2\parallel^2_{L^2(\mathbb{T}^4,g)}}\omega_2+d^+_{g^{\infty}}\gamma_1$.
 Since $P^+_{g_{(k,\lambda)}}(dx^1\wedge dx^3-dx^2\wedge dx^4)-d^+_{g_{(k,\lambda)}}\beta_{(k,\lambda)}$ is $d$-closed,
 \begin{eqnarray*}
   P^+_{g^{\infty}}(dx^1\wedge dx^3-dx^2\wedge dx^4)-d^+_{g^{\infty}}\beta_{\infty} &=& \frac{1}{\parallel\omega_2\parallel^2_{L^2(\mathbb{T}^4,g)}}\omega_2+d^+_{g^{\infty}}\gamma_1-d^+_{g^{\infty}}\beta_{\infty} \\
    &=&   \frac{1}{\parallel\omega_2\parallel^2_{L^2(\mathbb{T}^4,g)}}\omega_2+d^+_{g^{\infty}}\gamma_1-d^+_{g^{\infty}}\gamma_2
 \end{eqnarray*}
  is a $d$-closed $g^\infty$-self-dual $2$-form on $\mathbb{T}^4$.
  Moreover, we can get
   $d(d^+_{g^{\infty}}\gamma_1-d^+_{g^{\infty}}\gamma_2)=0$
   and
  $d^+_{g^{\infty}}\gamma_1=d^+_{g^{\infty}}\gamma_2$.
  By (\ref{equ +}), $d^+_{g^{\infty}}\gamma_1\wedge\omega_0\neq 0$,
   then $d^+_{g^{\infty}}\beta_{\infty}\wedge\omega_0=d^+_{g^{\infty}}\gamma_2\wedge\omega_0\neq 0$.
  This gives a contradiction to (\ref{supp7}).
   Hence, for any fixed $k\in\mathbb{N}$, there exists some $\lambda_0\in[1,+\infty)$ such that
   $$d^+_{g^{\lambda_0}_k}\beta_{(k,\lambda_0)}\wedge\omega_0\not\equiv 0$$ on $\{|\xi|\leq\frac{1}{k}\}$.
  This completes the proof of Claim \ref{claim}.


 \vskip 24pt

 \noindent Qiang Tan\\
 School of Mathematics and Statistics, Wuhan University, Wuhan, Hubei 430072, China
 (Current address: Faculty of Science, Jiangsu University, Zhenjiang, Jiangsu 212013, China)\\
 e-mail: tanqiang1986@hotmail.com\\

 \vskip 6pt

 \noindent Hongyu Wang\\
 School of Mathematical Sciences, Yangzhou University, Yangzhou, Jiangsu 225002, China\\
 e-mail: hywang@yzu.edu.cn\\

 \vskip 6pt

 \noindent Jiuru Zhou\\
 School of Mathematical Sciences, Yangzhou University, Yangzhou, Jiangsu 225002, China\\
 e-mail: zhoujr1982@hotmail.com

 \end{document}